\documentclass[12pt]{amsart}

\usepackage[mathscr]{eucal}
\usepackage{amscd}
\usepackage{amsfonts}
\usepackage{amsmath, amsthm, amssymb}
\usepackage{latexsym}
\usepackage[dvips]{graphics}

\theoremstyle{plain}
\newtheorem{thm}{Theorem}
\newtheorem{cor}[thm]{Corollary}

\newtheorem{prop}[thm]{Proposition}

\theoremstyle{definition}

\newtheorem{ex}[thm]{Example}

\theoremstyle{remark}
\newtheorem{rem}[thm]{Remark}

\numberwithin{equation}{section}

\title{WEIGHTED PROJECTIVE LINES AND RIEMANN SURFACES}

\author{Helmut Lenzing} % authors

\address{
\begin{flushleft}
        \hspace{0.3cm}  Institute of Mathematics \\
         \hspace{0.3cm}  University of Paderborn\\
         \hspace{0.3cm} Warburger Str.\ 100 \\
         \hspace{0.3cm}  33098 Paderborn,GERMANY\\
\end{flushleft}
}
\email{helmut@math.uni-paderborn.de}

\usepackage[all]{xy}
\begin{document}

%%%%%%%%%%%%%%%%%% Abstract Form %%%%%%%%%%%%%%%

\begin{abstract}
For the base field of complex numbers we discuss the relationship between categories of coherent sheaves on compact Riemann surfaces and categories of coherent sheaves on weighted smooth projective curves. This is done by bringing back to life an old theorem of Bundgaard-Nielsen-Fox proving Fenchel's conjecture for fuchsian groups.
\end{abstract}


% \medskip
%{\it $2000$ Mathematics Subject Classification{\rm :}}
%\quad Primary  18Fxx, 30F10;  Secondary 30F35, 14H60.
%%%%%%%%%%%%%%%%%%%%%%%%%%%%%%%%%

\maketitle

%%% HL-macros
\newcommand{\XX}{\mathbb{X}}
\newcommand{\YY}{\mathbb{Y}}
\newcommand{\Knull}[1]{\mathrm{K}_0(#1)}
\newcommand{\coh}[1]{\mathrm{coh}\,#1}
\newcommand{\cohnull}[1]{\mathrm{coh}_0#1}
\newcommand{\eqcoh}[2]{\mathrm{coh}_{#1}#2}
\newcommand{\Hom}[2]{\mathrm{Hom}(#1,#2)}
\newcommand{\Ext}[3]{\mathrm{Ext}^{#1}({#2},{#3})}
\newcommand{\euform}[2]{\langle #1,#2\rangle}
\newcommand{\aveuform}[2]{\langle\langle #1,#2\rangle\rangle}
\newcommand{\wt}[1]{\langle #1 \rangle}
\newcommand{\tw}[1]{[#1]}
\newcommand{\lcm}[1]{\mathrm{lcm}(#1)}
\newcommand{\vx}{\vec{x}}
\newcommand{\vy}{\vec{y}}
\newcommand{\vz}{\vec{z}}
\newcommand{\vc}{\vec{c}}
\newcommand{\vom}{\vec{\omega}}
\newcommand{\PSL}[2]{\mathrm{PSL}_{#1}(#2)}
\newcommand{\DD}{\mathbb{D}}
\newcommand{\LL}{\mathbb{L}}
\newcommand{\MM}{\mathbb{M}}
\newcommand{\PP}{\mathbb{P}}
\newcommand{\Oo}{\mathcal{O}}
\newcommand{\Ff}{\mathcal{F}}
\newcommand{\Kk}{\mathcal{K}}
\newcommand{\Hh}{\mathcal{H}}
\newcommand{\Kd}{\mathcal{K}}
\newcommand{\rk}[1]{\mathrm{rk}\,#1}
\newcommand{\dg}[1]{\mathrm{dg}\,#1}
\newcommand{\Aut}[1]{\mathrm{Aut}(#1)}
\newcommand{\NN}{\mathbb{N}}
\newcommand{\ZZ}{\mathbb{Z}}
\newcommand{\CC}{\mathbb{C}}
\newcommand{\HH}{\mathbb{H}}
\newcommand{\TT}{\mathbb{T}}
\newcommand{\eps}{\varepsilon}
\newcommand{\copr}{\sqcup}
\newcommand{\si}{\sigma}
\newcommand{\mmod}[1]{\mathrm{mod}(#1)}
\newcommand{\modgr}[2]{\mathrm{mod}^{#1}(#2)}
\newcommand{\modgrnull}[2]{\mathrm{mod}_0^{#1}(#2)}
\newcommand{\iso}{\cong}
\newcommand{\dual}[1]{D(#1)}
\newcommand{\al}{\alpha}
\newcommand{\be}{\beta}
\newcommand{\ga}{\gamma}
\newcommand{\la}{\lambda}
\newcommand{\wtilde}{\widetilde}

%%%%%%%%%%%%%%%%%%%%%%

\section{Introduction}
Throughout we work over the base field $\CC$ of complex numbers, though quite a number of results hold in larger generality. The central theme of this paper is weighted smooth projective curves (in particular weighted projective lines) and their relationship to compact Riemann surfaces, expressed by the following theorem.

\begin{thm}[Bundgaard-Nielsen, Fox]\label{thm:main}
With the exception of the weighted projective lines $\PP^1\wt{p,q}$ with $p\neq q$, there exists for each weighted smooth projective curve $\XX$ a compact Riemann surface $M$ and a finite subgroup $G$ of its automorphism group $\Aut{M}$ such that the category $\coh{\XX}$ of coherent sheaves on $\XX$ is equivalent to the category $\eqcoh{G}{M}$ of $G$-equivariant coherent sheaves on $M$, that is, to  the skew group category $(\coh{M})[G]$.
\end{thm}

Within the community of representation theory of finite dimensional algebras this relationship has been touched on in only isolated instances \cite[Example 5.8]{Geigle:Lenzing:1987}, \cite{Kirillov:2006}, \cite{Chen:Chen:Zhou:2015}, \cite{Chen:Chen:2017}. In this community the above general result is therefore largely unknown, though in other branches of mathematics the statement, expressed in a different language, is well known as \emph{Fenchel's conjecture} or \emph{Fox's theorem} stating that each fuchsian group (in Nielsen's terminology $F$-group) has a torsionfree normal subgroup of finite index. The proof of the conjecture started 1948 by a paper of Nielsen~\cite{Nielsen:1948}, continued 1951 with a paper by Bundgaard and Nielsen~\cite{Bundgaard:Nielsen:1951} with the final touch due to R.\ H.\ Fox~\cite{Fox:1952} one year later. Notice here the corrections by Chau~\cite{Chau:1983}. This happened at a time when the concepts referred to in the above theorem did hardly exist. Nowadays, Fox's theorem is further known to be a special case of \emph{Selberg's lemma}, stating that in characteristic zero each finitely generated matrix group has a torsionfree normal subgroup  of finite index~\cite{Alperin:1987}.

Weighted projective curves, and weighted projective lines in particular, appear in many different incarnations making them the meeting point of a remarkable variety of different theories:
\begin{itemize}
\item Smooth projective curves (compact Riemann surfaces), equipped with a weight function,
\item Compact holomorphic one-dimensional orbifolds~\cite{Thurston:2002}, \cite{Scott:1983},
\item Deligne-Mumford curves (stacks) \cite{Behrend:Noohi:2006}, \cite{Abdelgadir:Ueda:2015},
\item Hereditary noetherian categories with Serre duality \cite{Lenzing;1997}, \cite{Reiten:VandenBergh:2002}, \cite{Lenzing:Reiten:2006},
\item Fuchsian groups  and related (spherical, parabolic resp.\ hyperbolic) tessellations~\cite{Zieschang:Vogt:Coldewey:1980}, \cite{Lyndon:Schupp:1977},
\item Smooth projective curves (or Riemann surfaces) with a parabolic structure \cite{Seshadri:1982}, \cite{Lenzing:1998}, \cite{Crawley-Boevey:2010}.
\end{itemize}
Weighted projective lines in particular, are related by tilting theory to the representation theory of finite dimensional algebras~\cite{Geigle:Lenzing:1987}, \cite{Ringel:1984}. The above theorem extends this relationship to compact Riemann surfaces. General references to the topics treated in this article are \cite{Jones:Singerman:1987} and \cite{Shokurov:1998}. Concerning weighted projective curves we refer to \cite{Geigle:Lenzing:1987} and \cite{Reiten:VandenBergh:2002}. Concerning discrete groups we refer to \cite{Lyndon:Schupp:1977} and \cite{Zieschang:Vogt:Coldewey:1980}.

\section{The category of coherent sheaves}
As mentioned in the introduction, weighted Riemann surfaces (or weighted smooth projective curves) appear in many different contexts  and incarnations (function theory, algebraic geometry, orbifolds, stacks, quotients of tessellations, etc.). The bridge between the various contexts is formed by their categories  of coherent sheaves. We therefore discuss below how to express major properties and invariants of weighted Riemann surfaces (or weighted smooth projective curves) in terms of their categories of coherent sheaves.

To begin with, it is classical that there is a bijection between (isomorphism classes of) compact Riemann surfaces $M$ and smooth projective curves $X$ such that the category $\coh{M}$ of holomorphic coherent sheaves on $M$ is equivalent to the category $\coh{X}$ of algebraic coherent sheaves on $X$. The most important
invariant of $M$ (or $X$) is the \emph{function field} $K=\CC(M)$ (resp.\ $K=\CC(X)$), which may be defined by the equivalence of quotient (resp.\ module) categories
\begin{equation} \label{eq:functionfield}
\frac{\coh{M}}{\cohnull{M}}\iso \mmod{K},\textrm{ respectively } \frac{\coh{X}}{\cohnull{X}}\iso \mmod{K}.
\end{equation}
Here, $\cohnull{}$ refers to the Serre subcategory of coherent sheaves of finite length, and $\mmod{K}$ denotes the category of finite dimensional $K$-vector spaces. It is again classical~\cite{Chevalley:1951} that $M$ (or $X$) is determined up to isomorphism by the function field $K$, which is an \emph{algebraic function field in one variable} over $\CC$, that is, a finite algebraic extension of the rational function field $\CC(x)$ in one variable. Moreover, each algebraic function field of one variable over $\CC$ has the form $\CC(M)$ (equivalently $\CC(X)$). We further note~\cite{Chevalley:1951} that in the situation of \eqref{eq:functionfield} each isomorphism from $\CC(M)$ to $\CC(X)$ induces a bijection between the point sets underlying $M$ and $X$ which may be used to identify these sets.

Next we turn to the weighted setting. A \emph{weighted compact Riemann surface $\MM$} (or \emph{weighted smooth projective curve $\XX$}) is a pair $(M,w)$, or $(X,w)$, where $w$ is a \emph{weight function} on $M$ (resp.\ $X$). This is an integral valued function taking only values $w(x)\geq1$ and such that $w(x)>1$ holds for only finitely many points $x_1,x_2,\ldots,x_t$, called the \emph{weighted} or \emph{exceptional} points. The remaining points of $M$ or $X$ are called \emph{ordinary}. Two weighted Riemann surfaces (resp.\ weighted curves) $(M_1,w_1)$ and $(M_2,w_2)$ are called isomorphic if there exists an isomorphism from $M_1$ to $M_2$, commuting with the respective weight functions. Similarly, the automorphism group $\Aut{\MM}$ of a weighted Riemann surface $\MM=(M,w)$ or its corresponding weighted smooth projective curve $(X,w)$, consists of all (holomorphic) automorphisms of $M$, respectively all (algebraic) automorphisms of $X$  that are commuting with the weight function $w$. Also the automorphism group of $\MM$ can be expressed in terms of the category $\coh{\MM}$ as the subgroup of $\Aut{\coh{\MM}}$, the group of isomorphism classes of self-equivalences of $\coh\MM$, fixing the structure sheaf $\Oo$, compare~\cite{Lenzing:Meltzer:2000}.

Another way to think of a weighted Riemann surface $(M,w)$ is to view it as a \emph{holomorphic orbifold}, compare \cite{Thurston:2002}, \cite{Scott:1983}. For this, one chooses for each weighted point of $M$, say of weight $a$, a small open neighborhood $U$,  that is isomorphic to the open unit disk $\DD$, and replaces $U$ by the cone $\DD/\mu_a$, where the group $\mu_a$ of $a$th roots of unity acts on $\DD$ by multiplication. Clearly, we can switch back and forth between the two concepts without information loss. To a weighted Riemann surface $(M,w)$, resp.\ a weighted smooth projective curve $(X,w)$, is associated a category of coherent sheaves, that can be described in different ways. The fastest access to these categories is provided by applying for each weighted point the $p$-cycle construction from \cite{Lenzing:1998}  to the category $\coh{M}$  of holomorphic coherent sheaves on $M$, respectively to the category $\coh{X}$ of algebraic coherent sheaves on $X$. We obtain this way equivalent abelian hereditary noetherian categories $\coh\MM$  and $\coh\XX$ which have Serre duality and inherit the function field from $M$ (or $X$), extending the validity of formula \eqref{eq:functionfield} to the weighted setting. Having \emph{Serre duality} means, that there is a self-equivalence $\tau$ of $\coh\MM$ (or $\coh\XX$) such that $\dual{\Ext{1}{E}{F}}=\Hom{F}{\tau E}$ holds bifunctorially for all coherent sheaves $E$ and $F$, where $D$ stands for the usual vector space duality. It follows that $\coh\MM$, or $\coh\XX$, has almost-split sequences with $\tau$ serving as the Auslander-Reiten translation. Keeping track of the function field and analyzing the resulting tube structure for of coherent sheaves of finite length it is easy to see that $(M,w)$, or $(X,w)$, can be recovered from the respective category of coherent sheaves, compare~\cite{Geigle:Lenzing:1991}.

All the foregoing also holds if we take another option for constructing $\coh{\MM}$ (or $\coh{\XX}$) by expressing the weighted structure by means of a sheaf of hereditary orders, following Reiten and Van den Bergh \cite{Reiten:VandenBergh:2002}.  Again, we can switch back and forth between weighted compact Riemann surfaces, weighted smooth projectives curves, and hereditary noetherian categories with Serre duality (with an infinite dimensional function field). Note that, also in the weighted case, the function field always determines the underlying Riemann surface or smooth projective curve but not position and weights of the weighted points. These, however, are determined by the category $\coh{\MM}$ (or $\coh{\XX}$). From now on we only speak of weighted compact Riemann surfaces, leaving it to the reader to formulate the corresponding statements for weighted smooth projective curves.

Let $\MM=(M,w)$ be a weighted compact Riemann surface and $\coh\MM$ its category of coherent sheaves. The most important invariant of $\MM$, or $\coh\MM$, next to the function field $\CC(M)$, is the \emph{orbifold Euler characteristic}, or just Euler characteristic $\chi_\MM$ of $\MM$.
More specifically let $\MM=\MM\wt{a_1,a_2,\ldots,a_t}$, where the weight sequence $a_1,a_2,\ldots,a_t$ lists the weights $w(x_i)$ of the weighted points. (Notice that the above notation does not specify the actual position of the weighted points which hence should be obtained from the context.) The \emph{Euler form} is the bilinear form on the Grothendieck group $\Knull{\MM}$ of the category $\coh{\MM}$ of coherent sheaves on $\MM$ which is given on (classes of) coherent sheaves by the expression $\euform{X}{Y}=\dim \Hom{X}{Y}-\dim \Ext{1}{X}{Y}$. Let $\bar{a}=\lcm{a_1,a_2,\ldots,a_t}$. Then the \emph{averaged Euler form} is defined as
$$
\aveuform{X}{Y}=\frac{1}{\bar{a}}\cdot\sum_{i=0}^{\bar{a}-1}\euform{\tau^i X}{Y},
$$
where $\tau$ is the Auslander-Reiten translation of $\coh\MM$. We have two linear forms \emph{rank} $\rk$ and \emph{degree} $\dg$ on $\Knull{\MM}$ which are uniquely determined by the following properties:
\begin{enumerate}
\item For $X$ in $\coh{\MM}$ we always have $\rk{X}\geq0$, and the structure sheaf $\Oo$ has rank one. Moreover, $\rk{X}=0$ holds exactly if $X$ has finite length. Also $\rk{}$ is preserved under automorphisms of $\coh{\MM}$.
\item A non-zero sheaf $X$ of finite length has degree $>0$, and the structure sheaf has degree zero. Further each simple sheaf $S_x$, concentrated in a point $x$ of $\MM$, has degree $\frac{\bar{a}}{w(x)}$.
\end{enumerate}

By means of the averaged Euler form we obtain the following weighted form of a \emph{Riemann-Roch theorem}.
\begin{equation}
\aveuform{X}{Y}=\aveuform{\Oo}{\Oo}\cdot\rk{X}\rk{Y}+\frac{1}{\bar{a}}\left|\begin{array}{cc}\rk{X}&\rk{Y}\\ \dg{X}&\dg{Y}\end{array}\right|
\end{equation}
The expression $2\aveuform{\Oo}{\Oo}$ is called the \emph{orbifold Euler characteristic} $\chi_\MM$, or just the Euler characteristic of $\MM$. It is an important homological invariant of $\MM$ and has the accessible combinatorial description
\begin{equation}
\chi_\XX=2\cdot\wt{\Oo,\Oo}-\sum_{i=1}^t\left(1-\frac{1}{a_i}\right),
\end{equation}
where $\chi_M=2\cdot\wt{\Oo,\Oo}=2(1-g_M)$ is the Euler characteristic of the underlying Riemann surface $M$ and where $g_M=\dim{\Ext{1}{\Oo}{\Oo}}$  is the \emph{genus} of $M$.

If $G$ is a finite group of automorphisms of $\MM=(M,w)$, then there are only finitely many orbits $Gx$ in $M/G$ having a non-trivial stabilizer $G_x$ (which is necessarily cyclic). Putting $\bar{w}(Gx)=w(x)\cdot|G_x|$ defines $\MM/G=(M/G,\bar{w})$ as a weighted Riemann surface, called the \emph{orbifold quotient}, or just quotient, of $\MM$ by $G$. A classical theorem of Riemann and Hurwitz asserts:
\begin{thm}[Riemann-Hurwitz]
Let $\MM$ be a weighted compact Riemann surface and $G$ a finite subgroup of the automorphism group $\Aut{\MM}$ of $\MM$. Then $\MM/G$ is again a weighted compact Riemann surface having Euler characteristic
$$
\chi_{\MM/G}=\frac{\chi_\MM}{|G|}.
$$\qed
\end{thm}
In the above setting, the action of $G$ on $\MM$ induces an action of $G$ on $\coh{\XX}$, allowing to form the skew group category $(\coh{\MM})[G]$, compare~\cite{Reiten:Riedtmann:1985}, also known as the category $\eqcoh{G}{\MM}$ of \emph{$G$-equivariant coherent sheaves} on $\MM$.
\begin{prop}\label{prop:equiv}
If $G$ is a finite group of automorphisms of $\MM$ then $\coh{\MM/G}=\eqcoh{G}{\MM}$. \qed
\end{prop}
We illustrate the concepts by examples. Here, and in the sequel, the symbols $C_n$, $D_n$, $A_n$ and $S_n$ refer, respectively, to the cyclic group of order $n$, the dihedral group of order $2n$, the alternating group of degree $n$, and the symmetric group of degree $n$.
\begin{ex}
By definition a \emph{polyhedral group} $G$ is a finite subgroup of $\Aut{\PP^1}=\PSL{2}{\CC}$. A polyhedral group is either a cyclic group $C_n$ of order $n$, a dihedral group $D_n$ of order $2n$, a tetrahedral group $A_4$ of order $12$, an octahedral group $S_4$ of order $24$ or the icosahedral group $A_5$ of order $60$. The corresponding quotient $\PP^1/G$ is respectively the weighted projective line $\PP^1\wt{n,n}$, $\PP^1\wt{2,2,n}$, $\PP^1\wt{2,3,3}$, $\PP^1\wt{2,3,4}$ and $\PP^1\wt{2,3,5}$. In accordance with the Riemann-Hurwitz rule the corresponding Euler characteristics are respectively $2/n$, $1/n$, $1/6$, $1/12$ and $1/30$, respectively. For the above examples, where $\PP^1\wt{a,b,c}=\PP^1/G$ we therefore note the relationship
\begin{equation}
|G|=\frac{2}{1/a+1/b
+1/c-1}.
\end{equation}
\end{ex}
Another useful tool is the \emph{dominance graph} which is a directed graph whose vertices are (isomorphism classes of) weighted Riemann surfaces and where we draw an arrow $\xymatrix{\MM_1\ar[r]|-{G}&\MM_2}$ if there exists a finite subgroup $G$ of $\Aut{\MM_1}$ with $\MM_2\iso \MM_1/G$, where the label of the arrow should be interpreted as the isomorphism type of $G$.
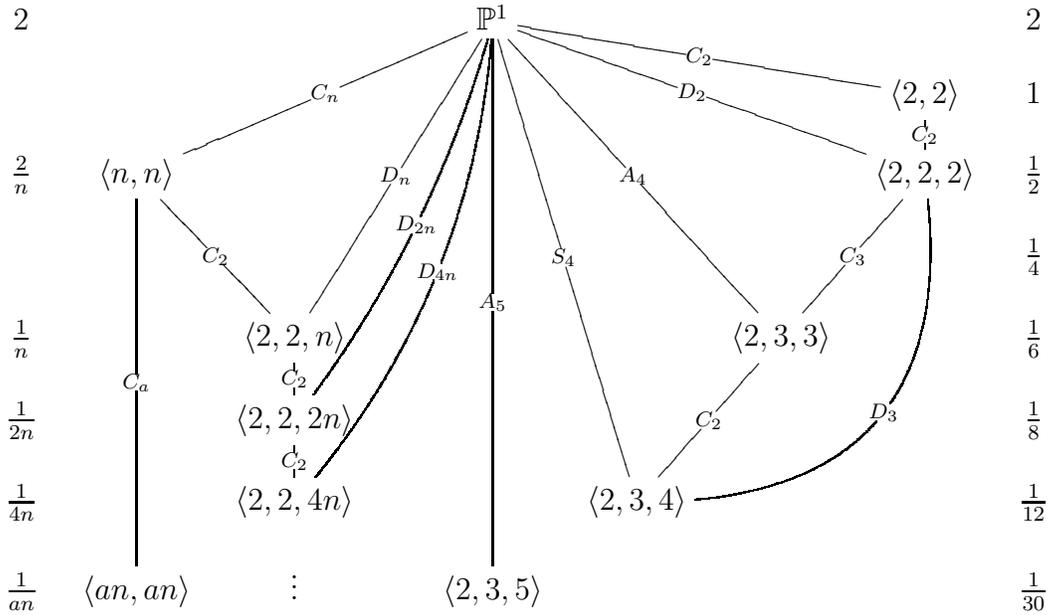
\begin{figure}[ht!]
  \centering
$$
\xymatrix@R10pt@C10pt{
2&          &      &&\PP^1\ar@{-}[drrr]|-{C_2}&& &                &2 \\
&          &      && &&       &\wt{2,2}\ar@{-}[d]|-{C_2}&1\\
\frac{2}{n}&\wt{n,n}\ar@{-}[rrruu]|-{C_n}&&&&&&\wt{2,2,2}\ar@/^4pc/@{-}[ddddll]|-{D_3}\ar@{-}[llluu]|-{D_2}\ar@{-}[ddl]|-{C_3}&\frac{1}{2}\\
&&         &      &                 &                &                  &&\frac{1}{4}\\
\frac{1}{n}&&\wt{2,2,n}\ar@{-}[luu]|-{C_2}\ar@{-}[uuuurr]|-{D_n}&&&  & \wt{2,3,3}\ar@{-}[lluuuu]|-{A_4}&                  &\frac{1}{6}\\
\frac{1}{2n}&   &\wt{2,2,2n}\ar@{-}[u]|-{C_2}\ar@/_0.7pc/@{-}[rruuuuu]|-{D_{2n}}&&      & & &         &\frac{1}{8}\\
\frac{1}{4n}&   &\wt{2,2,4n}\ar@{-}[u]|-{C_2}\ar@/_1.4pc/@{-}[rruuuuuu]|-{D_{4n}}&&&\wt{2,3,4}\ar@{-}[uuuuuul]|-{S_4}\ar@{-}[uur]|-{C_2}&&        &\frac{1}{12}\\
\frac{1}{an}&\wt{an,an}\ar@{-}[uuuuu]|-{C_a}&\vdots&& \wt{2,3,5}\ar@{-}[uuuuuuu]|-{A_5}& &                &                  &\frac{1}{30}}
$$
\caption{Dominance graph for positive Euler characteristic}
\label{fig:DominanceGraph}
\end{figure}
Figure~\ref{fig:DominanceGraph} shows the dominance graph (arrows top--down) for positive Euler characteristic, where a weight symbol $\wt{a,b,c}$ stands for the (isomorphism class of the) weighted projective line $\PP^1\wt{a,b,c}$. We have not included the weights $\wt{p,q}$ with $p\neq q$. Notice that we have added two scales for the Euler characteristic for members from the left part, where we assume $n\geq3$, (resp.\ from the right) part of the figure. We further note that $\PP^1\wt{2,3,4}$ and $\PP^1\wt{2,3,5}$ have trivial automorphism group, yielding terminal members $\wt{2,3,4}$ and $\wt{2,3,5}$ of the dominance graph. Also, it is immediate from the Riemann-Hurwitz theorem that for non-zero Euler characteristic there are no loops in the dominance graph. For Euler characteristic zero there are.

For negative Euler characteristic only few compact Riemann surfaces have a nontrivial automorphism group. Moreover the automorphism group is always finite:
\begin{prop}[Hurwitz]
Let $M$ be a compact Riemann surface of negative Euler characteristic. Then the automorphism group $\Aut{M}$ is finite, and its order is bounded by  $42\cdot|\chi_M|$.\qed
\end{prop}
Usually, the Hurwitz bound $84(g-1)$ is expressed in terms of the genus $g_M$ of $M$, where $g_M$ is related to the Euler characteristic for the identity $\chi_M=2(1-g_M)$. By definition, the Hurwitz bound is attained by the socalled \emph{Hurwitz surfaces} (or \emph{Hurwitz curves} as they are called in the context of smooth projective curves). The Hurwitz curve of smallest genus ($g=3, \chi=-4$) is Felix Klein's quartic $\Kk_4$ defined by the homogeneous polynomial $x^3y+y^3z+z^3x$. Its automorphism group is the simple group $G_{168}=\PSL{2}{7}$ of order $168$, \cite{Levy:1999}. Moreover, the quotient $\Kk_4/G_{168}$ is the weighted projective line $\XX=\PP^1\wt{2,3,7}$~\cite{Adler:1999} whose minimal value for the hyperbolic area $|\chi_\XX|=\frac{1}{42}$  is responsible for the Hurwitz bound. For further information on Hurwitz curves or surfaces see \cite{Kulkarni:1997}.

\section{Orbifold fundamental group and proof of Theorem~\ref{thm:main}}
To define the fundamental group of a weighted compact Riemann surface $\MM=(M,w)$ with weighted points $x_1,x_2,\ldots,x_t$ of weight $a_1,a_2,\ldots,a_t$, we need a modified version of homotopy for paths and loops which takes care of the weight function. For this we allow, relative to a fixed ordinary base point,  only paths possibly ending in a weighted point but not passing through any weighted point. We further select for each weighted point $x_i$ a loop $\si_i$ having winding index 1 with respect to $x_i$ and winding index 0 with respect to each other weighted point. Let $H$ be the monoid of homotopy classes of such paths. We next enlarge the homotopy relation to the smallest congruence relation on $H$ including the relations $\si_1^{a_1}=\si_2^{a_2}=\cdots=\si_t^{a_t}=1$. Restricting to the homotopy classes of loops, we obtain the \emph{orbifold fundamental group} $\pi_1^{orb}(\MM)$ of the weighted Riemann surface $\MM$.
\begin{prop}\label{prop:FundamentalGroup}
Let $\MM=\MM\wt{a_1,a_2,\ldots,a_t}$ be a weighted Riemann surface and $g$ the genus of the underlying Riemann surface. Then the orbifold fundamental group is the group on generators $$\al_1,\al_2,\ldots,\al_g,\be_1,\be_2,\ldots,\be_g,\si_1,\si_2,\ldots,\si_t$$ subject to the relations
$$
\si_1^{a_1}=\si_2^{a_2}=\cdots =\si_t^{a_t}=1=\si_1\si_2\cdots\si_t\,[\al_1,\be_1][\al_2,\be_2]\cdots[\al_g,\be_g],
$$
where $[a,b]$ denotes the commutator $aba^{-1}b^{-1}$ of $a$ and $b$.
\end{prop}
\begin{proof}
  It is classical that the fundamental group of a compact Riemann surface of genus $g$ is the group on generators $\al_1,\al_2,\ldots,\al_g,\be_1,\be_2,\ldots,\be_g$ subject to the relations $$[\al_1,\be_1][\al_2,\be_2]\cdots[\al_g,\be_g]=1.$$
  By definition of orbifold homotopy, the $t$ weighted points with weights $a_1,a_2,\ldots,a_t$ then yield the relations from the proposition.
\end{proof}

In the usual fashion, the space of all orbifold homotopy classes of paths in $\MM$ (starting at the base point and allowing weighted points only as end points), leads to the covering space $\wtilde{\MM}$, equipped with a natural topology and a projection map $\pi:\wtilde{\MM}\to \MM$ defining $\wtilde{\MM}$ as a branched cover of $\MM$ (with ramification over the weighted points). As usual $\wtilde{\MM}$ is simply connected and inherits from $\MM$ via $\pi:\wtilde{\MM}\to \MM$ a holomorphic structure. By construction, $\wtilde{\MM}$ is thus an ordinary (not necessarily compact) Riemann surface, and hence by the general Riemann mapping theorem, that is, the uniformization theorem of Poincar{\'e} and Koebe~\cite{De:Saint-Gervais:2016}, holomorphically isomorphic to either $\PP^1$, $\CC$ or $\HH$, where $\HH$ denotes the open upper complex half-plane.

\begin{thm} \label{thm:Uniformization}
Let $\MM$ be a weighted Riemann surface. Then the fundamental group $\pi_1^{orb}(\MM)$ acts on the universal orbifold cover $\wtilde{\MM}$ as group of deck transformations (the members of $\Aut\MM$ commuting with $\pi:\wtilde\MM\to \MM$).  This actions represents $\MM$ as orbifold quotient
  $$\MM={\wtilde{\MM}}/{\pi_1^{orb}(\MM)}.
  $$
Assuming, moreover, that $\MM$ is not isomorphic to any $\PP^1\wt{a_1,a_2}$ with $a_1\neq a_2$, we have the following trisection:
\begin{enumerate}
  \item \emph{spherical}: If $\chi_\MM>0$ then $\wtilde{\MM}=\PP^1$, and $\pi_1^{orb}(\MM)$ is a finite polyhedral group and $\MM$ is one of $\PP^1\wt{n,n}$, $\PP^1\wt{2,2,n}$, or $\PP^1\wt{2,3,a}$ with $a=2,3,5$.
  \item \emph{parabolic}: If $\chi_\MM=0$ then $\wtilde{\MM}=\CC$, and $\MM$ is either a smooth elliptic curve or a weighted projective line of tubular type $\wt{2,3,6}$, $\wt{2,4,4}$, $\wt{3,3,3}$ or $\wt{2,2,2,2,2}$.
  \item \emph{hyperbolic}: If $\chi_\MM<0$ then $\wtilde{\MM}=\HH$, and $\MM$ has hyperbolic type.
\end{enumerate}
\end{thm}
For the spherical (parabolic resp.\ hyperbolic) case the category $\coh\MM$ has tame domestic (tame tubular, resp.\ wild) representation type;
\begin{proof}
  The Riemann surfaces $\PP^1$, $\CC$ and $\HH$ admit K{\"a}hler metrics of constant curvature $+1$, $0$ and $-1$, respectively. This feature is inherited by the orbifold quotients $\MM$, see \cite{Shokurov:1998}. In view of the Gauss-Bonnet theorem, the curvature of $\MM$ is directly related to the orbifold Euler characteristic of $\MM$, yielding the trisection of the theorem.
\end{proof}

Assuming the above exclusion of certain weighted projective lines, Theorem~\ref{thm:Uniformization} implies that the groups appearing in Proposition~\ref{prop:FundamentalGroup} are just the (cocompact) fuchsian groups, that is the finitely generated subgroups $G$ of the automorphism groups $\Aut{X}$, for $X\in\{\PP^1,\CC,\HH\}$ with compact quotient $X/G$.

The following facts are well-known concerning fuchsian groups $G$ (or $F$-groups as they are often called):
\begin{enumerate}
  \item Finitely generated subgroups of fuchsian groups are again fuchsian.
  \item With the notation from Proposition~\ref{prop:FundamentalGroup} each torsion element of $G$ is conjugate to a power of some $\si_i$, $i=1,\ldots,t$.
\end{enumerate}

The statement of the next proposition is also known as \emph{Fenchel's conjecture}.
\begin{prop}[Bundgaard-Nielsen, Fox]\label{prop:Fox}
Each fuchsian group has a torsionfree normal subgroup of finite index.
\end{prop}
\begin{proof}
  We sketch the proof. Using the notation of Proposition~\ref{prop:FundamentalGroup} let $G$ be a fuchsian group. An elementary induction argument shows, see Mennicke's elegant paper \cite{Mennicke:1967,Mennicke:1968}, that it suffices to prove the claim for the factor group of $G$ modulo the normal subgroup generated by $\al_1,\ldots,\al_g$ and $\be_1,\ldots,\be_g$. The assertion thus reduces to deal with the orbifold fundamental group $G$ of a weighted projective line $\PP^1\wt{a_1,a_2,\ldots,a_t}$. Passing further to the factor group of $G$ by the normal subgroup generated by $\si_{4},\ldots,\si_t$, the claim reduces to the case of a \emph{triangle group} $\TT(a_1,a_2,a_3)=\wt{\si_1,\si_2,\si_3\mid \si_1^{a_1}=\si_2^{a_2}=\si_3^{a_3}=\si_1\si_2\si_3=1}$, that is, the orbifold fundamental group of $\PP^1\wt{a_1,a_2,a_3}$. In \cite{Fox:1952}, Fox constructs two finite permutations $c_1$ of order $a_1$ and $c_2$ of order $a_2$ such that $c_1c_2$ has order $a_3$. Putting $c_3=(c_1c_2)^{-1}$ yields an obvious homomorphism $h:\TT(a_1,a_2,a_3)\to G$, sending $\si_i$ to $c_i$, ($i=1,2,3$), where $G$ is the group generated by $c_1$ and $c_2$. Hence the kernel $N$ of $h$ has finite index in $\TT(a_1,a_2,a_3)$. Since $\si_1,\si_2´,\si_3$ keep their orders under $h$, and since each torsion element from $\TT(a_1,a_2,a_3)$ is conjugate to some power of $\si_1$, $\si_2$ or $\si_3$, the kernel of $h$ contains no nontrivial torsion.
\end{proof}

\begin{cor}\label{cor:uniformization}
  Each weighted compact Riemann surface $\MM$ as in Theorem~\ref{thm:Uniformization} is isomorphic to the orbifold quotient $M/G$ for a compact Riemann surface $M$ and a finite subgroup $G$ of $\Aut{M}$.
\end{cor}
\begin{proof}
  Let $G=\pi_1^{orb}{\MM}$. By Theorem~\ref{thm:Uniformization}, $\MM$ is isomorphic to the orbifold quotient $\wtilde\MM/G$. Let $N$ be a normal subgroup of finite index having no torsion elements, then $N$ acts on $\wtilde\MM$ without fixed point, yielding therefore a quotient $M=\wtilde\MM/N$ that is an ordinary Riemann surface. Putting $H=G/N$, the finite group $H$ acts on $M$ with quotient $M/H=(\wtilde{\MM}/N)/(G/N)=\wtilde{\MM}/G=\MM$. Since $\MM$ is compact and $H$ is finite, also the Riemann surface $M$ is compact.
\end{proof}
In view of Proposition~\ref{prop:equiv},  Corollary~\ref{cor:uniformization} finishes the proof of Theorem~\ref{thm:main}.

\begin{ex}
By means of a current computer algebra system it is easy to determine finite permutation groups satisfying the relations of a given triangle group $\TT(a_1,a_2,a_3)$. Here are a few examples:
\begin{itemize}
\item The permutations $(1, 2)(3, 6)$ and $(1, 2, 3, 4, 5, 6, 7)$ of order 2 and 7 have a product $(1, 3, 7)(4, 5, 6)$ of order $3$. The generated permutation group $G$ is the unique simple group $G_{168}=\PSL{2}{7}$ of order $168$. By Corollary~\ref{cor:uniformization} there hence exists a compact Riemann surface $M$ establishing $\PP^1\wt{2,3,7}$ as the orbifold quotient $M/G$. Invoking the Riemann-Hurwitz formula, we deduce that $M$ has Euler characteristic $-\frac{1}{42}\cdot 168=-4$ (and genus 3). And indeed, Klein's quartic $xy^3+yz^3+zx^3$ has Euler characteristic $-3$ and automorphism group $G_{168}$ and leads to the quotient $\PP^1\wt{2,3,7}$, see \cite{Adler:1999}.
\item The permutations $(3, 4)(5, 7)$ and $(1, 2, 3, 4, 5, 6, 7)$ of order 2 and 7 have a product $(1, 2, 3, 5)(6, 7)$ of order 4; they generate a group, that is again the simple group $G_{168}$. By Corollary~\ref{cor:uniformization} this realizes $\PP^1\wt{2,4,7}$ as the orbifold quotient $M/G_{168}$, where $M$ is a compact Riemann surface of Euler characteristic $-18$ (or genus 10) that can be identified as the Hessian determinant curve associated to Klein's quadric, see~\cite{Adler:1999}.
  \item The permutations $(1, 4)(2, 6)(3, 7)(5, 8)$ and $(1, 2, 3, 4, 5, 6, 7, 8, 9)$ have order 2 and 9 with a product $(1, 5, 9)(2, 7, 4)(3, 8, 6)$ of order 3. These permutations generate a simple group $G$ of order $504$ and thus yield a realization $M/G$ of $\PP^1\wt{2,3,9}$ as the orbifold quotient of a compact Riemann surface $M$ of Euler characteristic $-28$ (or genus $15$).
\end{itemize}
\end{ex}
Nowadays there exist several alternatives to Fox's proof of Fenchel's conjecture. While Fox's proof is based on finite permutation groups, there is a proof due to Mennicke~\cite{Mennicke:1967,Mennicke:1968} using $3\times 3$-matrices with entries in an algebraic number field determined by the three weights $\wt{a,b,c}$. Another proof, attributed to Macbeath, using $2\times2$-matrices over finite fields, appears in the book of Zieschang, Vogt and Coldewey~\cite{Zieschang:Vogt:Coldewey:1980}, compare also the article of Feuer~\cite{Feuer:1971}. Finally, in view of Theorem~\ref{thm:Uniformization}, Proposition~\ref{prop:Fox} turns out to be a special case of Selberg's lemma~\cite{Alperin:1987}, stating that  in characteristic zero finitely generated matrix groups have normal torsionfree subgroups of finite index. We also mention that there is a vast literature about representing weighted Riemann surfaces, in particular weighted projective lines, as orbifold quotient of compact Riemann surfaces of (preferably small) genus. For further information on this topic we refer to \cite{Breuer:2000} and the literature quoted there.

\section{Realization techniques}
In this section we collect a number of general results, originating in the theory of weighted projective lines, supporting to identify a given weighted projective line as an orbifold quotient of a weighted projective curve, given by explicit homogeneous equations.

Let $\XX=\PP^1\wt{a_1,a_2,\ldots,a_t}$ be a weighted projective line. We recall~\cite{Geigle:Lenzing:1987} that the projective coordinate algebra $S$ of $\XX$ is the quotient of the polynomial algebra $k[x_1,\ldots,a_t]$ by the ideal generated by the canonical equations $x_i^{a_i}=x_2^{a_2}-\la_i x_1^{a_1}$, $i=3,\ldots,t$. Here, the $\la_i$ are supposed to be non-zero and pairwise distinct. The algebra $S$ is graded by the rank-one abelian group $\LL=\LL(a_1,\ldots,a_t)$ on generators $\vx_1,\vx_2,\ldots,\vx_t$ subject to the relations $a_1\vx_1=a_2\vx_2=\cdots=a_t\vx_t$. This $\LL$-graded algebra $S$ yields by Serre construction the category of coherent sheaves on $\XX$, see \cite{Geigle:Lenzing:1987}.

By means of the degree homomorphism $\delta:\LL(a_1,a_2,\ldots,a_t)\to \ZZ$, $\vx_i\mapsto \lcm{a_1,\ldots,a_t}/a_i$, the algebra $S$ can alternatively be viewed to be $\ZZ$-graded. By $\ZZ$-graded Serre construction we then obtain a weighted projective curve $\YY=\YY\tw{a_1,a_2,\ldots,a_t}$ such that
$$
\coh\YY=\frac{\modgr{\ZZ}{S}}{\modgrnull{\ZZ}{S}}.
$$
We call $\YY$ the \emph{twisted companion} of $\XX$.

\begin{prop}[\cite{Lenzing:2016}] \label{prop:twisted:companion}
In characteristic zero, the twisted companion $\YY=\YY\tw{a_1,a_2,\ldots,a_t}$ of the weighted projective line $\XX=\PP^1\wt{a_1,a_2,\ldots,a_t}$ is a weighted smooth projective curve with the following properties where $\bar{a}=\lcm{a_1,a_2,\ldots,a_t}$.
\begin{enumerate}
\item $\XX$ is isomorphic to the orbifold quotient $\YY/G$, where $G=\frac{\mu_{a_1}\times\mu_{a_2}\times \cdots \times \mu_{a_t}}{\mu_{\bar{a}}}$.
\item If the degrees $|x_i|=\delta(\vx_i)$, $i=1,\ldots,t$, are pairwise coprime, then $\YY$ is a (nonweighted) smooth projective curve.\qed
\end{enumerate}
\end{prop}
\begin{ex}
\begin{enumerate}
\item Prime examples for nonweighted twisted companion curves are the smooth elliptic curves $\YY\tw{2,3,6}$, $\YY\tw{2,4,4}$, $\YY\tw{3,3,3}$ and $\YY\tw{2,2,2,2;\lambda}$ corresponding to the four tubular weight types.
\item A hyperbolic example is given by the smooth projective companion curve $\YY\tw{2,6,6}$ of Euler characteristic $-2$ (genus 2), where the defining polynomial $x_1^2+x_2^6+x_3^6$ is graded by giving $(x_1,x_2,x_3)$ degrees $(3,1,1)$.
\item Also important, and classical, are the \emph{Fermat curves} $\Ff_a=\YY\tw{a,a,a}$, where the variables of the defining polynomial $x_1^a+x_2^a+x_3^a$ all get degree one. $\Ff_a$ has Euler characteristic $-a(a-3)$.
\item Slightly more general, the companion curves $\YY\tw{a,a,\ldots,a}$, for $t\geq2$ identical weights concentrated in $t$ pairwise distinct points of $\PP^1$, yield smooth projective curves of Euler characteristic $-a^{t-2}((t-2)a-t)$.
\end{enumerate}
\end{ex}
For $t$ identical weights we abbreviate $\PP^1\wt{a,a,\ldots,a}$ and $\YY\tw{a,a,\ldots,a}$ by $\PP^1\wt{a^{[t]}}$ and $\YY\tw{a^{[t]}}$, respectively. Similarly, for a finite subset $A$ of $\PP^1$ where each member of $A$ obtains weight $a$, we write $\PP^1\wt{a^{[A]}}$ for the corresponding weighted projective line and $\YY\tw{a^{[A]}}$ for its twisted companion curve.

Let $P$ be a polyhedral group. Then $P$ is either the cyclic group $C_n$ of order $n\geq1$, or the dihedral group of order $2n$, $n\geq2$ or else the tetrahedral group $A_3$ of order $12$, the octahedral group $A_4$ of order $24$ or the icosahedral group $A_5$ of order $60$. The last three cases we call \emph{platonic} and let $P_3$, $P_4$ and $P_5$ be, respectively, the tetrahedral, octahedral and icosahedral group. We note that $P_n$ has order $\frac{12n}{6-n}$.

Our next result produces quite a number of weighted projective lines, represented as an orbifold quotients $M/G$, where $M$ is a smooth projective curve given by an explicit system of canonical equations. We say that the resulting weighted projective lines and their weight types have \emph{polyhedral type}.

\begin{thm}[Polyhedral Symmetries] \label{thm:polyhedral}
Let $\eps_1,\eps_2,\eps_3\in\{0,1\}$, $a$ an integer $\geq1$ and $r$ an integer $\geq0$. Let $P$ be a finite subgroup of $\Aut{\PP^1}$. We distinguish the following cases:
\begin{enumerate}
  \item \emph{The cyclic case $P=C_n$:} There exists a $P$-stable subset $A$ of $\PP^1$ of cardinality $|A|=\eps_1+\eps_2+r$ such that $\PP^1\wt{a^{[A]}}/P=\PP^1\wt{a^{\eps_1}n,a^{\eps_2}n,a^{[r]}}$.
  \item\emph{The dihedral case $P=D_n$:} There exists a $P$-stable subset $A$ of $\PP^1$ of cardinality $|A|=2\eps_1+n(\eps_2+\eps_3)+2nr$ such that $\PP^1\wt{a^{[A]}}/P=\PP^1\wt{2\cdot a^{\eps_1},2\cdot a^{\eps_2},n\cdot a^{\eps_3},a^{[r]}}$.
  \item \emph{The platonic case $P=P_n$:} For each $n\in\{3,4,5\}$ there exists a $P$-stable subset $A=A_n$ of cardinality $|A|=\frac{12n}{6-n}\cdot\left(\frac{\eps_1}{2}+\frac{\eps_2}{3}+\frac{\eps_3}{n}+r\right)$ with quotient $\PP^1\wt{a^{[A]}}/P=\PP^1\wt{2a^{\eps_1},3a^{\eps_2},na^{\eps_3},a^{[r]}}$.
\end{enumerate}
Moreover, in each of the above cases the following holds:

The twisted companion $\YY=\YY\tw{a^{[A]}}$ of $\PP^1\wt{a^{[A]}}$ is a smooth projective curve. The $P$-action on $\PP^1\wt{a^{[A]}}$ lifts to a $P$-action on $\YY$, inducing on $\YY$ an action of $G=\frac{\mu_a^b}{\mu_a}\rtimes P$ with quotient $\YY/G=\PP^1\wt{a^{[A]}}/P$.
\end{thm}

\begin{proof} We only deal with the platonic case, where $P=P_n$ with $n\in\{3,4,5\}$. Then $\PP^1/P_n$ has three orbits $Pz_1$, $Pz_2$ and $Pz_3$ with stabilizer groups $P_{z_1}=C_2$, $P_{z_2}=C_3$ and $P_{z_3}=C_n$. All other points $y\in \PP^1\setminus (Pz_1\copr Pz_2 \copr Pz_3)$ have trivial stabilizer group. A finite $P$-stable subset $A$ of $\PP^1$ hence has the form
  \begin{equation}
A=\eps_1(P.z_1)\copr \eps_2(P.z_2) \copr \eps_3(P.z_3) \copr\left(P.y_1\copr\cdots\copr P.y_r \right)
\end{equation}
for $\eps_1,\eps_2,\eps_3\in\{0,1\}$, $r\geq0$ and distinct points $y_1,y_,\ldots,y_r$ from the complement of $Pz_1\copr Pz_2 \copr Pz_3$. If each point of $A$ is given weight $a$ ($a\geq1)$, then the orbifold quotient of $\PP^1\wt{a^{[A]}}$ is the weighted projective line $\PP^1\wt{2a^{\eps_1},3a^{\eps_2},na^{\eps_3},a^{[r]}}$. Observing that $\PP^1\wt{a^{[A]}}$ is the quotient of $\YY\tw{a^{[A]}}$ by the group $\frac{\mu_a^{|A|}}{\mu_a}$ then proves the claim.
\end{proof}
\begin{cor} \label{cor:polyhedral}
For integers $n\geq2$, $a\geq 1$ we have the following realizations
of weighted projective lines as quotients of smooth projective curves:

\begin{minipage}{0.4\textwidth}
\begin{eqnarray*}
 \PP^1\wt{n,n,a} &=&\YY\tw{a^{[n]}}/\frac{\mu_a^{[n]}}{\mu_a}\rtimes C_n,\\
 \PP^1\wt{2,2a,n} &=& \YY\tw{a^{[n]}}/\frac{\mu_a^n}{\mu_a}\rtimes D_n,\\
 \PP^1\wt{3,3,2a} &=& \YY\tw{a^{[6]}}/\frac{\mu_a^6}{\mu_a}\rtimes A_4,\\
\end{eqnarray*}
\end{minipage}
\begin{minipage}{0.4\textwidth}
\begin{eqnarray*}
\PP^1\wt{n,a,an} &=& \YY\tw{a^{[n]},a}/\frac{\mu_a^{n+1}}{\mu_a}\rtimes C_n, \\
\PP^1\wt{3,4,2a} &=& \YY\tw{a^{[12]}}/\frac{\mu_a^{12}}{\mu_a}\rtimes S_4,\\
 \PP^1\wt{3,5,2a} &=& \YY\tw{a^{[30]}}/\frac{\mu_a^{30}}{\mu_a}\rtimes A_5. \qed\\
\end{eqnarray*}
\end{minipage}
\end{cor}
We illustrate Theorem~\ref{thm:polyhedral} and Corollary~\ref{cor:polyhedral} by Table~\ref{tab:arnold} discussing the realizability of the entries from Arnold's strange duality list, consisting of the 14  weight triples (socalled Dolgachev numbers) yielding a fuchsian singularity that is a hypersurface. Whenever possible, we have included a realization of polyhedral type. The table lists for each of the $14$ weight types $\wt{a,b,c}$ a realization of $\XX=\PP^1\wt{a,b,c}$ as a quotient $M/G$ for a compact Riemann surface (smooth projective curve) $M$, and a realization of $M$ by equations.

By $\Ff_n$ we denote the \emph{Fermat curve} $x^n+y^n+z^n$, where $x,y,z$ get degree one. Further, $\Kk_4$ denotes Felix Klein's quartic curve, and correspondingly
$G_{168}=\PSL{2}{7}$ denotes the (unique) simple group of order $168$. That $G_{168}$ acts on $\Kk_4$ with quotient $\PP^1\wt{2,3,7}$ was proved by Klein in 1879, see the book \cite{Levy:1999} which is devoted entirely to the various aspects of Klein's curve. In particular, we refer to \cite[Prop.\ 12.1 and 12.2]{Adler:1999} showing the realization of $\wt{2,4,7}$ through the Hessian determinant $5x^2y^2z^2-(xy^5+yz^5+zx^5)$ of $\Kk_4$. We note that with the exception of $\wt{2,3,7}$ all weight types from the list have a polyhedral realization. Note in this context that polyhedral realizations tend to have a high genus!

\begin{rem}
We may merge Theorem~\ref{thm:polyhedral} with Proposition~\ref{prop:twisted:companion} to achieve a generalization of Theorem~\ref{thm:polyhedral}. With the notations of the theorem we may endow each point $x_i$ of the finite $P$-stable subset $A$ of $\PP^1$ by an individual weight $a_i$ such that the $a_i$ are pairwise coprime, instead giving all $x_i$ the same weight $a$. This leads to the concept of \emph{generalized polyhedral type}. An example is the first realization for weight type $\wt{2,4,6}$ from Table~\ref{tab:arnold}.
\end{rem}
\small
\begin{table}[ht!]
\def\hl{\hline}
\renewcommand{\arraystretch}{1.5}
 \caption{Strange duality weights}
 \begin{tabular}{|c|c|c|c|c|c|c|}\hl
weights   & $G$         & $|G|$ & $-\chi_\XX$& $-\chi_M$& $g_M$& curve $M$/equations\\\hl
$\wt{2,3,7}$&$G_{168}$    &168    &$\frac{1}{42}$& $4$& 3 &$\Kk_4: x^3y+y^3z+z^3x$\\\hl
$\wt{2,3,8}$&$\frac{\mu_4^3}{\mu_4}\rtimes D_3$&$96$&$\frac{1}{24}$&4&3&$\Ff_4:x^4+y^4+z^4$ \\\hl
$\wt{2,3,9}$&$\frac{\mu_3^4}{\mu_e}\rtimes A_4$&$396$&$\frac{2}{3}$&$18$&$10$& $\YY\tw{3,3,3,3}$\\\hl
$\wt{2,4,5}$&$\frac{\mu_2^5}{\mu_2}\rtimes D_5$&$160$&$\frac{1}{20}$&$8$&$5$&$\YY\tw{2,2,2,2,2}$\\\hl
$\wt{2,4,6}$&$\frac{\mu_3\times\mu_6^2}{\mu_6}\rtimes D_2$&$72$&$\frac{1}{12}$ &$6$&$4$& $\YY\tw{3,6,6}$\\\cline{2-7}
           &$\frac{\mu_2^6}{\mu_2}\rtimes D_6$&$384$&$\frac{1}{12}$&$32$&$17$&
           $\YY\tw{2,2,2,2,2,2}$\\\hl
$\wt{2,4,7}$&$G_{168}$&$168$&$\frac{3}{28}$&$18$&$10$&$5x^2y^2z^2-(xy^5+yz^5+zx^5)$\\\cline{2-7}
   &$\frac{\mu_2^7}{\mu_2}\rtimes D_7$&$2^7\cdot7$&$\frac{3}{28}$&$96$&$49$&$\YY\tw{2,2,2,2,2,2,2}$\\\hl
$\wt{2,5,5}$&$\frac{\mu_2^5}{\mu_2}\rtimes C_5$&$80$&$\frac{1}{10}$&$8$&5&$\YY\tw{2,2,2,2,2}$\\\hl
$\wt{2,5,6}$&$\frac{\mu_3^5}{\mu_3}\rtimes D_5$&$810$&$\frac{2}{15}$&$108$&$55$&$\YY\tw{3,3,3,3,3}$\\\hl
$\wt{3,3,4}$&$\frac{\mu_4^3}{\mu_4}\rtimes C_3$&$48$&$\frac{1}{12}$&$4$&$3$&$\Ff_4:x^4+y^4+z^4$\\\hl
$\wt{3,3,5}$&$\frac{\mu_5^3}{\mu_5}\rtimes C_3$&$75$&$\frac{2}{15}$&$10$&$6$&$\Ff_5:x^5+y^5+z^z$\\\hl
$\wt{3,3,6}$&$\frac{\mu_6^3}{\mu_6}\rtimes C_3$&$108$&$\frac{1}{6}$&$18$&$10$&$\Ff_6:x^6+y^6+z^6$\\\hl
$\wt{3,4,4}$&$\frac{\mu_3^4}{\mu_3}\rtimes C_4$&$108$&$\frac{1}{6}$&$18$&$10$&$\YY\tw{3,3,3,3}$\\\hl
$\wt{3,4,5}$&$\frac{\mu_2^{30}}{\mu_2}\rtimes A_5$&$2^{29}\cdot60$&$\frac{13}{60}$&$13\cdot2^{29}$&$13\cdot2^{28}+1$&$\YY\tw{2^{[30]}}$\\\hl
$\wt{4,4,4}$&$\frac{\mu_4^3}{\mu_4}$&$16$&$\frac{1}{4}$&$4$&$3$&$\Ff_4:x^4+y^4+z^4$\\\hl
 \end{tabular}\label{tab:arnold}
 \end{table}
 \normalsize

\bibliographystyle{amsplain}
%%%\bibliography{../Lit/math_bibfile,../Lit/math_bib}

\end{document}